\documentclass{amsart}
\usepackage{amssymb,amsthm}
\usepackage{bbm} % for the complex imaginary unit i
\usepackage{enumerate}
\usepackage[dvips]{graphics}
\usepackage{lscape}
\usepackage{color}
\usepackage{rotating}
\usepackage{hyperref}
\usepackage[T1]{fontenc}
\usepackage[latin1]{inputenc}
\usepackage{ae}
\usepackage{aecompl}

\newtheorem{theorem}{Theorem}[section]
\newtheorem{lemma}[theorem]{Lemma}
\newtheorem{proposition}[theorem]{Proposition}

\newtheorem{conjecture}[theorem]{Conjecture}
\theoremstyle{definition}
\newtheorem{definition}[theorem]{Definition}
\newtheorem{remark}[theorem]{Remark}
\newtheorem{example}[theorem]{Example}

\newcommand{\N}{\mathbb N}

\newcommand{\R}{\mathbb R}

\renewcommand{\l}{\ell}

\renewcommand\arraystretch{1.1}%

\newcommand{\tabfrac}{\frac}

\newcommand{\ep}[0]{\varepsilon}

\newcommand{\ph}[0]{\varphi}

\newcommand{\la}[0]{\lambda}

\newcommand{\Th}[0]{\Theta}

\newcommand{\Si}[0]{\Sigma}

\newcommand{\x}{\bar X}
\newcommand{\ax}{\langle\x\rangle} % adjoin associating variables
\newcommand{\axy}{\langle X,Y\rangle} % adjoin two associating variables X, Y
\newcommand{\csim}{\stackrel{\mathrm{cyc}}{\thicksim}}

\DeclareMathOperator{\tr}{tr}
\DeclareMathOperator{\sym}{Sym}
\DeclareMathOperator\Span{span}

\definecolor{light-gray}{gray}{0.55}
\definecolor{white}{gray}{1}

\newcommand{\dd}[2]{\genfrac{}{}{0pt}{}{#1}{#2}}

\title[Sums of hermitian squares and the BMV conjecture]
{Sums of hermitian squares\\and the BMV conjecture}

\author{Igor Klep}
\address{Igor Klep, Univerza v Ljubljani, Oddelek za matematiko
In\v stituta za matematiko, fiziko in mehaniko,
Jadranska 19, 1111 Ljubljana, Slovénie}
\email{igor.klep@fmf.uni-lj.si}
\thanks{The first author acknowledges the financial support from the state
budget by the Slovenian Research Agency (project No. Z1-9570-0101-06).}
\author{Markus Schweighofer}
\address{Markus Schweighofer, Universit\'e de Rennes 1,
Laboratoire de Mathématiques,
Campus de Beaulieu,
35042 Rennes cedex,
France}
\thanks{Supported by the DFG grant ``Barrieren''.}
\email{markus.schweighofer@univ-rennes1.fr}
\date{July 31, 2008}

\begin{document}
\begin{abstract}
We show that all the coefficients of the polynomial
$$\tr((A+tB)^m)\in\R[t]$$
are nonnegative whenever $m\le 13$ is a nonnegative integer
and $A$ and $B$ are positive
semidefinite matrices of the same size. This has previously been known
only for $m\le 7$. The validity of the statement for arbitrary $m$ has
recently been shown to be equivalent to the Bessis-Moussa-Villani conjecture
from theoretical physics. In our proof, we establish a connection to sums of
hermitian squares of polynomials in noncommuting variables and to
semidefinite programming.
As a by-product we obtain an example of a real polynomial in two
noncommuting variables having nonnegative trace on all symmetric
matrices of the same size, yet not being a sum of hermitian squares
and commutators.
\end{abstract}
\maketitle

\section{Introduction}

While
attempting to simplify the calculation of partition functions
in quantum statistical mechanics, Bessis, Moussa and Villani (BMV) conjectured
in 1975 \cite{bmv} that for any hermitian $n\times n$ matrices
$A$ and $B$ with $B$ positive semidefinite, the function
$$
\ph^{A,B}:\R\to\R,\quad t\mapsto\tr\left(e^{A-tB}\right)
$$
is the Laplace transform of a positive measure $\mu^{A,B}$ on 
$\R_{\geq 0}$. That is,
$$
\ph^{A,B}(t)=\int_0^\infty e^{-tx}\,d\mu^{A,B}(x)
$$
for all $t\in\R$.
By Bernstein's theorem, this is equivalent to $\ph^{A,B}$ being
completely monotone, i.e.,
$$
(-1)^s \frac{d^s}{dt^s}\ph^{A,B}(t) \geq 0
$$
for all $s\in\N_0$ and $t\in\R_{\geq 0}$.

Due to its importance (cf. \cite{bmv,ls}) there
is an extensive literature on this conjecture. Nevertheless
it has resisted all attempts at proving it. 
For an overview of all the approaches before 1998 leading to
partial results, we refer the reader to Moussa's survey \cite{mou}.

In 2004, Lieb and Seiringer \cite{ls} achieved a breakthrough 
paving the way to a series of new attempts at proving the BMV conjecture. 
They succeeded in restating the conjecture
in the following purely algebraic form:

\begin{conjecture}[BMV, algebraic form]\label{bmv}
The polynomial \[p:=\tr((A+tB)^m)\in\R[t]\] has only nonnegative coefficients
whenever $A$ and $B$ are $n\times n$ positive semidefinite matrices.
\end{conjecture}

The coefficient of $t^k$ in $p$ is the trace of $S_{m,k}(A,B)$, the sum of all
words of length $m$ in $A$ and $B$ in which B appears exactly $k$ times
(and therefore $A$ exactly $m-k$ times).
It is easy to see that these coefficients are real for hermitian $A,B$.

Suppose $A,B$ are positive semidefinite $n\times n$ matrices.
For $k\le 2$ or $m-k\le 2$, each
word appearing in $S_{m,k}(A,B)$ has nonnegative trace as is easily seen.
This proves the conjecture for $m\le 5$. For $n\le 2$, $A$ can (as always)
be assumed to be diagonal and after a diagonal change of basis also $B$ has
only nonnegative entries. Hence the conjecture is trivial for $n\le 2$.
The first nontrivial case $(m,k,n)=(6,3,3)$ was verified by Hillar and Johnson
\cite{hj} with the help of a computer algebra system by considering entries
of both $3\times 3$ matrices, $A$ and $B$, as scalar and therefore 
\emph{commuting} variables. H\"agele \cite{hag} shifted the focus 
from scalars to symbolic computation with matrices (regardless of
their size) and gave a surprisingly simple argument  
settling the case $(m,k)=(7,3)$ and thus also
$(m,k)=(7,4)$ by symmetry. Combined with the easy observations from above, 
this proves Conjecture \ref{bmv} for $m=7$.

H\"agele then deduced the case $m=6$, which 
he could not solve directly with his technique, 
by appealing to the following seminal result due to Hillar \cite{hil}: 
If Conjecture \ref{bmv} is true for $m$, then it is
also true for all $m'<m$ \cite[Corollary 1.8]{hil}.
A strengthening \cite[Theorem 1.7]{hil} of this result (see Section
\ref{outline} for a precise statement)
is crucial for our main contribution:

\begin{theorem}\label{bmv9}
The {\rm BMV}\! Conjecture {\rm \ref{bmv}} holds for $m\le 13$.
\end{theorem}

We exploit semidefinite programming to find certain certificates for
nonnegativity of $\tr(S_{m,k}(A,B))$ which are dimensionless (i.e., valid
for all $n$). These certificates are algebraic identities in the ring
of polynomials in two \emph{noncommuting} variables involving sums of hermitian
squares. The found identities are exact though obtained with the help of
numerical computations. But they exist only for certain pairs $(m,k)$ and we
have to rely on Hillar's work to deduce Theorem \ref{bmv9}. For instance,
such a sum of
hermitian squares certificate does not exist for $(m,k)=(6,3)$, see Example
\ref{non-sos}.

With the benefit of hindsight, H\"agele's argument can be read
as such a certificate for the case $(m,k)=(7,3)$. 
However, the certificates we give for $(m,k)=(14,4)$ and
$(m,k)=(14,6)$ are much more involved and seem to be impossible 
to find by hand.

This paper is organized as follows. Section \ref{m2s} 
develops the appropriate algebraic framework needed 
for the desired nonnegativity certificates.
In Section \ref{s2m} the existence of such a certificate
is transformed into a linear matrix inequality (LMI) 
enabling us to search for these certificates using semidefinite
programming (SDP).
Section \ref{outline} explains the overall argument for the
proof of Theorem \ref{bmv9}. The proof
itself is presented in full detail in Section \ref{export}.
A synopsis of our results
and other recent developments
is given in Section \ref{conclude},
where we also relate the BMV conjecture
to another just as old open problem of Connes on II$_1$ factors. 
Finally, in the appendix we streamline the proof of the mentioned
crucial result of Hillar and 
give an alternative argument to prove
the BMV conjecture for $m=13$ avoiding Hillar's theorem.

\section{From matrices to symbols}\label{m2s}

The gist of our method is to model the matrices as
\emph{noncommuting} variables instead of disaggregating them
into scalar entries modeled by \emph{commuting} variables.
To this end we introduce the ring of polynomials in two
noncommuting variables.

\begin{remark}
It is easy to see \cite[Lemma 3.15]{ksconnes} that the
nonnegativity of $\tr(S_{m,k}(A,B))$ for all positive semidefinite
\emph{complex}
$A$ and $B$ of all sizes need only be checked for all positive semidefinite
(in particular symmetric) \emph{real} $A$ and $B$ of all sizes 
(by identifying $n\times n$ complex matrices with $2n\times 2n$ real
matrices).
We therefore
work over the real numbers.
\end{remark}

We write $\axy$ for the monoid freely generated by $X$ and $Y$, i.e.,
$\axy$ consists of \emph{words} in two letters (including the empty word
denoted by $1$). Let $\R\axy$ denote the associative $\R$-algebra freely
generated by $X$ and $Y$. The elements of $\R\axy$
are polynomials in the noncommuting variables $X$ and $Y$ with coefficients
in $\R$. An element of the form $aw$ where $0\neq a\in\R$ and
$w\in\axy$ is called a \emph{monomial} and $a$ its
\emph{coefficient}. Hence words are monomials whose coefficient is
$1$. We endow $\R\axy$ with the involution $p\mapsto p^*$
fixing $\R\cup\{X,Y\}$ pointwise. Recall that an \emph{involution}
has the properties $(p+q)^*=p^*+q^*$, $(pq)^*=q^*p^*$
and $p^{**}=p$ for all $p,q\in\R\axy$. In particular, for each word
$w\in\axy$, $w^*$ is its reverse.

\begin{definition}
Two polynomials
$f,g\in\R\axy$ are called \emph{cyclically equivalent} ($f\csim g$)
if $f-g$ is a sum of commutators in $\R\axy$. Here elements of
the form $pq-qp$ are called \emph{commutators} ($p,q\in\R\axy$).
\end{definition}

This definition reflects the fact that $\tr(AB)=\tr(BA)$ for
square matrices $A$ and $B$ of the same size.
The following proposition shows that cyclic equivalence can easily be checked and
will be used tacitly in the sequel. Part (c) is a special case of
\cite[Theorem 2.1]{ksconnes} motivating the definition of cyclic equivalence.

\begin{proposition}\label{cyceqrem}
\begin{enumerate}[\rm (a)]
\item \label{cyceqword}
For $v,w\in\axy$, we have $v\csim w$ if and only if there are
$v_1,v_2\in\axy$ such that $v=v_1v_2$ and $w=v_2v_1$.
\item\label{cyceqalg}
Two polynomials $f=\sum_{w\in\axy}a_ww$ and $g=\sum_{w\in\axy}b_ww$
$(a_w,b_w\in\R)$ are cyclically equivalent if and only if for each $v\in\axy$,
$$\sum_{\dd{w\in\axy}{w\csim v}}a_w=\sum_{\dd{w\in\axy}{w\csim v}}b_w.$$
\item\label{trace0} Suppose $f\in\R\axy$ and $f^*=f$. Then $f\csim 0$ if and only if $\tr(f(A,B))=0$
for all real symmetric matrices $A$ and $B$ of the same size.
\end{enumerate}
\end{proposition}

\begin{definition}
For each subset $S\subseteq\R\axy$, we introduce the set
$$\sym S:=\{g\in S\mid g^*=g\}$$
of its \emph{symmetric elements}. Elements of the form $g^*g$
($g\in\R\axy$) are called \emph{hermitian squares}. We denote by 
$$
\Si^2:=\{ \sum_i g_i^*g_i\mid g_i\in\R\axy\}\subseteq\sym\R\axy
$$
the convex cone of all sums of hermitian squares and by 
\begin{align*}
\Th^2& :=\{ f\in\R\axy\mid \exists g\in\Si^2:\, f\csim g\}\\
&\;=\Si^2+\{ \sum_i (g_ih_i-h_ig_i)\mid g_i,h_i\in\R\axy\} \subseteq\R\axy
\end{align*}
the convex cone
of all polynomials that are cyclically equivalent to a sum of 
hermitian squares.
\end{definition}

The following theorem proved in \cite{hel} also holds for several variables
and motivates the use of sums of hermitian squares (see \cite{hp} for a survey
of recent developments). We will only use
the easy implication from \eqref{helton-sos} to \eqref{helton-matrix}.

\begin{theorem}[Helton]\label{helton}
The following are equivalent for $f\in\sym\R\axy$:
\begin{enumerate}[{\rm(i)}]
\item $f\in\Si^2$; \label{helton-sos}
\item \label{helton-matrix}
$f(A,B)$ is positive semidefinite for all $n\in\N$ and
$A,B\in\sym\R^{n\times n}$.
\end{enumerate}
\end{theorem}

To obtain the desired type of certificates we try to merge
Proposition \ref{cyceqrem}\eqref{trace0} with Theorem \ref{helton}.
However, such certificates do not always exist.

\begin{remark}\label{helton-trace}
Consider the following conditions for $f\in\R\axy$: 
\begin{enumerate}[{\rm(i)}]
\item $f\in\Th^2$; \label{hel-tr1}
\item \label{hel-tr2}
$\tr (f(A,B))\geq 0$ for all $n\in\N$ and
$A,B\in\sym\R^{n\times n}$.
\end{enumerate}
Then
\eqref{hel-tr1} implies \eqref{hel-tr2} but not vice versa.
For instance, $$YX^4Y+XY^4X-3XY^2X+1\in\sym\R\axy$$ satisfies
\eqref{hel-tr2} but not \eqref{hel-tr1} (see \cite[Example 4.4]{ksconnes} for 
details).
Later on we will see further such examples.
\end{remark}

\section{From symbols to matrices}\label{s2m}

To search systematically for the certificates just introduced,
we develop a \emph{noncommutative} version of the Gram matrix method.
The corresponding theory for polynomials in \emph{commuting} variables
is well-known and has been studied and used extensively, see 
e.g.~\cite{clr,ps}.

Checking whether a polynomial in
noncommuting variables is an element of $\Si^2$ or $\Th^2$, respectively, 
is most efficiently done via the so-called \textit{Gram matrix method}.
Given a symmetric $f\in\R\axy$ of degree $\leq 2d$ and a vector $\bar v$ containing 
all words in $X,Y$ of degree $\leq d$, there is a real symmetric matrix $G$ with
$f=\bar v^* G \bar v$.
(Here $\bar v^*$ arises from $\bar v$ by applying the involution entrywise to the
transposed vector $\bar v^t$.) Every such matrix $G$ is called a \emph{Gram matrix}
for $f$. Obviously, the set of all Gram matrices for $f$ is an affine subspace.

\begin{example}\label{runex}
Consider the polynomial $$h:=X^4+2XYX+2X^2+Y^2+2Y+1\in\sym\R\axy.$$
Since $h$ has degree four, we choose
$$\bar v:=[1,X,Y,X^2,XY,YX,Y^2]^t.$$
Then every Gram matrix for $h$ has the form
$$G=\begin{bmatrix}\ 1\ &0&1&\ a\ &\ 0\ &\ 0\ &\ b\ \\
                   0&2-2a&0&0&0&1&0\\
                   1&0&1-2b&0&0&0&0\\
                   a&0&0&1&0&0&0\\
                   0&0&0&0&0&0&0\\
                   0&1&0&0&0&0&0\\
                   b&0&0&0&0&0&0
                   \end{bmatrix}\in\sym\R^{7\times 7}.$$
We will revisit this example below.
\end{example}

From Cholesky's decomposition we deduce that $f\in\sym\R\axy$ is a sum of 
hermitian squares if and only
if it has a positive semidefinite Gram matrix. Indeed, if $G=C^*C$ is a positive semidefinite
Gram matrix for $f$, then $f=\bar v^*C^*C\bar v=(C\bar v)^*(C\bar v)=\sum_ig_i^*g_i\in\Si^2$ where
$g_i\in\R\axy$ is the $i$-th entry of the vector $C\bar v$. The converse follows the same line of
reasoning.

\medskip\noindent
{\bf Example \ref{runex} continued.} There is no positive semidefinite Gram matrix $G$ for $h$
since the determinant of the submatrix
$$\begin{bmatrix}G_{22}&G_{26}\\G_{62}&G_{66}\end{bmatrix}=\begin{bmatrix}2-2a&1\\1&0\end{bmatrix}$$
is always negative. Hence $h\not\in\Si^2$.

\medskip
The existence of a sum of hermitian squares decomposition of $f\in\sym\R\axy$ is equivalent
to an LMI feasibility problem. As such it can be decided by solving the SDP 
$$\text{minimize\ }\tr(G)\quad\text{subject to\quad $\bar v^*G\bar v=f$, $G$ positive semidefinite.}$$
Note that $\bar v^*G\bar v=f$ are just linear constraints on the entries of $G$ as one sees
by comparing coefficients. The objective function $G\mapsto\tr(G)$ is often a good choice for
finding nice low rank matrices $G$ but can be replaced by any other function linear in the entries
of $G$. If the polynomial is dense (no sparsity), the dimension of the 
LMI
is equal to $(2^{d+1}-1)\times (2^{d+1}-1)$. For more on SDP,
we refer the reader to the survey \cite{tod}.

Likewise, checking whether $f\in\Th^2$ can be done by solving the SDP
$$\text{minimize\ }\tr(G)\quad\text{subject to\quad $\bar v^*G\bar v\csim f$, $G$ positive
semidefinite.}$$
By Proposition \ref{cyceqrem}\eqref{cyceqalg}, $\bar v^*G\bar v\csim f$ are again linear constraints on
the entries of $G$.

For the sake of convenience, from now on a real symmetric matrix $G$  will be called a \emph{Gram matrix} for $f\in\R\axy$ (with respect to a vector of words $\bar v$) if $f\csim \bar v^*G\bar v$.

\medskip\noindent
{\bf Example \ref{runex} continued.} Every Gram matrix (in the new sense) for $h$ has the form
$$
\begin{bmatrix}
1 & 0 & 1 & 1-\frac{1}{2} a_1&
-a_2-a_3 & a_2 &
\frac{1}{2}-\frac{1}{2} a_4 \\
0 & a_1 & a_3 & 0 &
-a_6-a_7+1 & a_6 &
-a_8-a_9 \\
1 & a_3 & a_4 & a_7 & a_8
& a_9 & 0 \\
1-\frac{1}{2} a_1 & 0 & a_7 & 1 &
-a_{10} & a_{10} & -\frac{1}{2}
a_{11}-\frac{1}{2} a_{12}\\
-a_2-a_3 & -a_6-a_7+1 &
a_8 & -a_{10} & a_{11} & 0 &
-a_5 \\
a_2 & a_6 & a_9 & a_{10} & 0
& a_{12} & a_5 \\
\frac{1}{2}-\frac{1}{2} a_4 &
-a_8-a_9 & 0 & -\frac{1}{2}
a_{11}-\frac{1}{2} a_{12} & -a_5 &
a_5 & 0
\end{bmatrix}.
$$
Setting $a_4=a_7=1$ and all other $a_i$ to zero, we get the positive semidefinite matrix
$G=\begin{bmatrix}1&0&1&1&0&0&0\end{bmatrix}^*\begin{bmatrix}1&0&1&1&0&0&0\end{bmatrix}$
with corresponding representation $h\csim (X^2+Y+1)^2\in\Si^2$, i.e., $h\in\Th^2$.

\medskip
In the proof of our main result we will use the Gram matrix method
to show that certain $S_{m,k}(X^2,Y^2)\in\Th^2$.
We start by dramatically reducing the sizes of corresponding SDPs with a
monomial reduction. For this, we need a technical lemma.

\begin{lemma}\label{trrad} Let 
$p_i\in\R\axy$.
\begin{enumerate}[{\rm (a)}]
\item If for $A,B\in\sym\R^{n\times n}$, $\tr\left(\sum_i (p_i^*p_i)(A,B)\right)=0$, then $p_i(A,B)=0$ for all $i$.\label{trradloc}
\item If $\sum_ip_i^*p_i\csim 0$, then $p_i=0$ for all $i$.\label{trradglob}
\end{enumerate}
\end{lemma}

\begin{proof}
(a) Denote by $e_j$ the canonical basis vectors of $\R^n$. Then
$$0=\tr(\sum_i(p_i^*p_i)(A,B))=\sum_{i,j}\langle(p_i^*p_i)(A,B)e_j,e_j\rangle\\
=\sum_{i,j}\langle p_i(A,B)e_j,p_i(A,B)e_j\rangle.$$
Hence $p_i(A,B)e_j=0$ for all $i,j$ and thus $p_i(A,B)=0$ for all $i$.\\
(b) If $\sum_ip_i^*p_i\csim 0$, then $\tr(\sum_ip_i(A,B)^*p_i(A,B))=0$, and by the above,
$p_i(A,B)=0$ for all symmetric $A$ and $B$ of all sizes $n$. This implies $p_i=0$ for all
$i$ (see e.g.~\cite[Proposition 2.3]{ksnirgends}).
\end{proof}

Not only do we drastically reduce the number of words needed in the Gram method
for $S_{m,k}(X^2,Y^2)$ but we also impose a block structure on the Gram matrix $G$
with blocks $G_i$. This is done in the following proposition. We use
self-explanatory notation like $\{X^2,Y^2\}^\ell$ for the set of all
words that are concatenations of $\ell$ copies of $X^2$ and $Y^2$.

\begin{proposition}\label{wred} Fix $m,k\in\N$.
\begin{enumerate}[\rm (a)]
\item\label{GramEE} If $m$ and $k$ are even, set
\begin{align*}
V_1 &:= \left\{v\in\{X^2,Y^2\}^{\frac m2} \mid \deg_X v=m-k,\,\deg_Y v=k\right\}, \\
V_2 &:= \left\{v\in X \{X^2,Y^2\}^{\frac m2-1} X \mid \deg_X v=m-k,\,\deg_Y v=k\right\}, \\
V_3 &:= \left\{v\in Y \{X^2,Y^2\}^{\frac m2-1} Y \mid \deg_X v=m-k,\,\deg_Y v=k\right\}.
\end{align*}
\item\label{GramOE} If $m$ is odd and $k$ is even, set
\begin{align*}
V_1 &:= \left\{v\in X \{X^2,Y^2\}^{\frac{m-1}2}  \mid \deg_X v=m-k,\,\deg_Y v=k\right\}, \\
V_2 &:= \left\{v\in  \{X^2,Y^2\}^{\frac{m-1}2} X \mid \deg_X v=m-k,\,\deg_Y v=k\right\}.
\end{align*}
\item\label{GramOO} If $m$ and $k$ are odd, set
\begin{align*}
V_1 &:= \left\{v\in Y \{X^2,Y^2\}^{\frac{m-1}2}  \mid \deg_X v=m-k,\,\deg_Y v=k\right\}, \\
V_2 &:= \left\{v\in  \{X^2,Y^2\}^{\frac{m-1}2} Y \mid \deg_X v=m-k,\,\deg_Y v=k\right\}.
\end{align*}
\item\label{GramEO} If $m$ is even and $k$ is odd, set
\begin{align*}
V_1 &:= \left\{v\in X \{X^2,Y^2\}^{\frac m2-1} Y \mid \deg_X v=m-k,\,\deg_Y v=k\right\}, \\
V_2 &:= \left\{v\in Y \{X^2,Y^2\}^{\frac m2-1} X \mid \deg_X v=m-k,\,\deg_Y v=k\right\}.
\end{align*}
\end{enumerate}
Let $\bar v_i$ denote the vector $[v]_{v\in V_i}$. Then
$S_{m,k}(X^2,Y^2)\in \Th^2$
if and only if there exist positive semidefinite matrices
$G_i\in\sym\R^{V_i\times V_i}$ such that
\begin{equation}\label{gram}
S_{m,k}(X^2,Y^2)\csim\sum_i\bar v_i^*G_i\bar v_i.
\end{equation}
If $G_i=C_i^*C_i$ and
$C_i\in\R^{J_i\times V_i}$ $(J_i$ some index set$)$, then with
$[p_{i,j}]_{j\in J_i}:=C_i\bar v_i$ we have
\begin{equation}\label{smksos}
S_{m,k}(X^2,Y^2)\csim\sum_{i,j}p_{i,j}^*p_{i,j}.
\end{equation}
\end{proposition}

\begin{proof}
The second statement is clear since
$$\sum_i\bar v_i^*G_i\bar v_i=\sum_i\bar v_i^*C_i^*C_i\bar v_i=
  \sum_i(C_i\bar v_i)^*C_i\bar v_i=\sum_{i,j}p_{i,j}^*p_{i,j}.$$
  
We assume without loss of generality that $1\le k\le m-1$.
Suppose that $S_{m,k}(X^2,Y^2)\in\Th^2$, i.e.,
\begin{equation}\label{keyeq}
S_{m,k}(X^2,Y^2)\csim\sum_jp_j^*p_j
\end{equation}
for finitely many $0\neq p_j\in\R\axy$. Set $d:=\max_j\deg_Yp_j$ and let $P_j$
be the sum of all monomials of degree $d$ with respect to $Y$ appearing in
$p_j$.

Fix real symmetric matrices $A$ and $B$ of the same
size. For any real $\la$, we have
$\la^{2k}\tr(S_{m,k}(A^2,B^2))=\tr(\sum_jp_j(A,\la B)^*p_j(A,\la B))$.
We consider this as an equality of real polynomials in $\la$.

If we assume
$d>k$, then $\tr(\sum_jP_j(A,B)^*P_j(A,B))=0$ since the degree of the right
hand side polynomial cannot exceed the degree of the left hand side polynomial.
By \eqref{trradloc} of Lemma \ref{trrad}, we get $P_j(A,B)=0$ for all $j$.
Since $A$ and $B$ were arbitrary, this implies $P_j=0$ 
by Lemma \ref{trrad}\eqref{trradglob},
contradicting the choice of $d$.
Therefore all monomials appearing in $p_j$
have degree $\leq k$ in $Y$. By similar arguments, one shows that all
$p_j$ are actually homogeneous of degree $m-k$ in $X$ and homogeneous
of degree $k$ in $Y$, i.e., $p_j\in\Span_\R W$ where $W$ is the set of
all words of length $m$ with the letter $X$ appearing $m-k$ times and the
letter $Y$ appearing $k$ times.

\smallskip
{\it Claim.} Suppose we are in one of the cases (a)--(d) and $v_i\in V_i$
for each $i$. Then $v_i^*v_j\csim u$ for some $u\in\{X^2,Y^2\}^m$ if and only if $i=j$.

{\it Proof of claim.} The ``if'' part is immediate. To show the ``only if'' part, we
assume that $i\neq j$ and show that $v_i^*v_j$ contains
$YX^\ell Y$ or $XY^\ell X$ as a subword for some odd $\l$. Then the claim follows
by Proposition \ref{cyceqrem}\eqref{cyceqword}.

The existence of such a subword must
be checked case by case. As an example, consider (a). By symmetry arguments, it suffices
to look at $v_1^*v_2$ and $v_2^*v_3$. In the former case, the letter 
at position $m+1$ in $v_1^*v_2$ is an $X$ which is followed to the left
and right hand side by finitely many $X^2$. This block of $X$'s has odd length
and is embraced at both ends by a $Y$ since we have assumed $k\ge 1$.
In the latter case, there is an $X$ at the $m$-th and a $Y$ at the $(m+1)$-st
position in $v_2^*v_3$. This $Y$ is followed to the right hand side by finitely many $Y^2$ giving a
block of $Y$'s of odd length surrounded by $X$'s.

The other cases (b)--(d) are essentially the same, proving the claim.

\smallskip
Write each $p_j$ as $p_j=\sum_ip_{i,j}+q_j$ where $p_{i,j}\in\Span_\R V_i$
and $q_j\in\Span_\R U$ with $U:=W\setminus\bigcup_iV_i$. By the claim,
$p_j^*p_j=\sum_ip_{i,j}^*p_{i,j}+r_j$ where 
$\sum_ip_{i,j}^*p_{i,j}$ is
a linear combination of words that are cyclically equivalent to a word
in $\{X^2,Y^2\}^m$ and $r_j$ is in the linear span of
words not cyclically equivalent to a word
in $\{X^2,Y^2\}^m$. By part \eqref{cyceqalg} of Proposition \ref{cyceqrem},
it follows that \eqref{keyeq} can be split into
$$S_{m,k}(X^2,Y^2)\csim\sum_{i,j}p_{i,j}^*p_{i,j}\qquad
\text{and}\qquad 0\csim\sum_jr_j.$$

Now let $J$ be the index set consisting of all $j$ and define matrices
$C_i\in\R^{J\times V_i}$ by $[p_{i,j}]_{j\in J}=C_i\bar v_i$. Then the matrices
$G_i:=C_i^*C_i$ are positive semidefinite and satisfy \eqref{gram}.
\end{proof}

We illustrate the proposition by two examples.

\begin{example}\label{sos84}
We have $S_{8,4}(X^2,Y^2)\in\Th^2$. For instance, with
\begin{align*}
\bar v_1&=[Y^2  X^2  Y^2  X^2,\, Y^4  X^4,\, X^2  Y^4  X^2,\,Y^2  X^4  Y^2,\, X^4  Y^4,\, X^2  Y^2  X^2  Y^2]^t, \\
\bar v_2&=[X  Y^4  X^3,\, X  Y^2  X^2  Y^2  X,\, X^3  Y^4  X]^t, \\
\bar v_3&=[Y^3  X^4  Y,\, Y  X^2  Y^2  X^2  Y,\, Y  X^4  Y^3]^t
\end{align*}
and
$$
G_1=\begin{bmatrix}
4&4&0&3&1&1\\
4&4&0&3&1&1\\
0&0&3&0&3&3\\
3&3&0&3&0&0\\
1&1&3&0&4&4\\
1&1&3&0&4&4\\
\end{bmatrix}, \quad
G_2=G_3=\begin{bmatrix}
1&0&-1\\
0&0&0\\
-1&0&1\\
\end{bmatrix},
$$
$S_{8,4}(X^2,Y^2)\csim\sum_{i=1}^3\bar v_i^*G_i\bar v_i$. 
The matrices $G_i$ which we found using SDP are positive semidefinite
as can be seen from their characteristic polynomials
\begin{align*}
p_{G_1}&=-108 t^3 + 129 t^4 - 22 t^5 + t^6\in\R[t], \\
p_{G_2}=p_{G_3}&=2 t^2 - t^3\in\R[t].
\end{align*}
Alternatively, we can use the Cholesky
decompositions $G_i=C_i^*C_i$ for
$$
C_1=\frac 12 \begin{bmatrix}
4&4&0&3&1&1\\
0&0&2\sqrt 3&0&2 \sqrt 3&2 \sqrt 3\\
0&0&0&\sqrt 3&-\sqrt 3&-\sqrt 3\\
\end{bmatrix}, \quad
C_2=C_3=\begin{bmatrix}
1&0&-1\\
\end{bmatrix}.$$ 
\end{example}

A first nontrivial nonnegativity certificate of this type was found in an ad hoc fashion
by Hägele \cite{hag}, namely
\begin{equation}\label{sos73}
\begin{split}
S_{7,3}(X^2,Y^2)\csim\,&
7(Y^2X^4Y)^*(Y^2X^4Y)+ \\& 7(X^2Y^2X^2Y+X^4Y^3)^*(X^2Y^2X^2Y+X^4Y^3)\in\Si^2.
\end{split}
\end{equation}
This proves Conjecture \ref{bmv} for $m=7$ (since the cases $k\le 2$ and $m-k\le 2$ are trivial
and $S_{7,4}(X^2,Y^2)=S_{7,3}(Y^2,X^2)\in\Th^2$). Note that the representation \eqref{sos73}
uses only words from $V_1$ of Proposition \ref{wred}\eqref{GramOO}. Hägele also showed that there
is no such representation for $S_{6,3}(X^2,Y^2)$ using only words from $V_1$ of Proposition
\ref{wred}\eqref{GramEO}. However, he speculated that admitting more words might lead to such
a representation meaning in our setup that $S_{6,3}(X^2,Y^2)\in\Th^2$. Our next example proves
that this is not the case.

\begin{example}\label{non-sos}
We show that $S_{6,3}(X^2,Y^2)\not\in\Th^2$.
Suppose, by way of contradiction, that $S_{6,3}(X^2,Y^2)\in\Th^2$.
Then by Proposition \ref{wred}\eqref{GramEO}, with
the basis
$$
V=\{Y^3X^3,\, YX^2Y^2X,\, XY^2X^2Y,\, X^3Y^3\}
$$
we can find a positive semidefinite Gram matrix for $S_{6,3}(X^2,Y^2)$ that is block diagonal of the form
$$
G_{6,3}=\left[\begin{array}{cccc}
a_{11}&a_{12}&0&0\\
a_{12}&a_{22}&0&0\\
0&0&b_{11}&b_{12}\\
0&0&b_{12}&b_{22}
\end{array}\right]\in\R^{4\times 4}.
$$
With $\bar v=[v]_{v\in V}$, it follows from $S_{6,3}(X^2,Y^2)\csim\bar v^* G_{6,3}\bar v$ 
that 
$$
G_{6,3}=\left[\begin{array}{cccc}
a_{11}&a_{12}&0&0\\
a_{12}&a_{22}&0&0\\
0&0&2-a_{22}&6-a_{12}\\
0&0&6-a_{12}&6-a_{11}
\end{array}\right].
$$
For a positive semidefinite matrix of this form,
$0\leq a_{11}\leq 6$, $0\leq a_{22}\leq 2$,
\begin{eqnarray}
\label{eq63-1}
a_{12}^2 &\leq& a_{11}a_{22},\\
\label{eq63-2}
(6-a_{12})^2&\leq& (6-a_{11})(2-a_{22}).
\end{eqnarray}
By adding \eqref{eq63-1} and \eqref{eq63-2}, we obtain
\begin{equation*}
36-12a_{12}+2a_{12}^2\leq 12-2a_{11}-6a_{22}+2a_{11}a_{22}.
\end{equation*}
As $-2a_{11}-6a_{22}+2a_{11}a_{22}= a_{22}(a_{11}-6)+a_{11}(a_{22}-2)\leq 0$,
this implies
$$
0\ge a_{12}^2-6a_{12}+12=(a_{12}-3)^2+3,
$$
a contradiction. Hence $S_{6,3}(X^2,Y^2)\not\in \Th^2$.
\end{example}

\section{Strategy of the proof}\label{outline}

An important ingredient in the proof of Theorem \ref{bmv9} will be 
the following descent result of Hillar \cite[Theorem 1.7]{hil}:

\begin{theorem}[Hillar]
\label{hillar}
The failure of 
Conjecture {\rm \ref{bmv}} for a certain $(m,k)$ implies failure for all $(m',k')$ with
$m'-k'\ge m-k$ and $k'\ge k$.
\end{theorem}

In view of this theorem it suffices to prove Conjecture \ref{bmv}
for $(m,k)=(14,4)$ and $(m,k)=(14,6)$. To do this we apply our Gram matrix
method to prove that $S_{14,4}(X^2,Y^2)\in\Th^2$ and $S_{14,6}
(X^2,Y^2)\in\Th^2$.

Since the search for positive semidefinite Gram matrices is done by SDP,
the entries of the found matrices are only floating point numbers and
do not provide a sound proof for the existence of a certificate of
nonnegativity. However, in our case, there happen to exist
such Gram matrices with \emph{rational} entries and we have employed several
strategies and heuristics to find them.

First, we have detected symmetries and
patterns in the numerical solutions and imposed them
as additional constraints in subsequent SDPs. Second, we have worked with
different objective functions in order to find solutions with some
``nice'' rational entries that could be fixed. Finally, we have
employed rounding techniques involving heuristics to guess the prime 
factors appearing
in the denominators of the presumably rational entries. 
All too often, we have however lost
numerical stability and had to backtrack in this manually guided
refinement process.

For a systematic treatment of finding exact rational sum of squares certificates
for polynomials in \emph{commuting} variables we refer the reader to \cite{pp},
see also \cite{hi} and the references therein.

\section{Proof of Theorem {\rm\ref{bmv9}}}\label{export}

\noindent As mentioned above, 
it suffices to show that
$S_{14,4}(X^2,Y^2), S_{14,6}(X^2,Y^2)\in\Th^2$ (cf.~the table on page
\pageref{triangle} below).
Let
\begin{align*}
\bar v_{14,4}=&[Y^2X^{10}Y^2,\, X^4Y^2X^2Y^2X^4,\, X^6Y^4X^4,\, X^2Y^2X^6Y^2X^2,\,
X^4Y^2X^4Y^2X^2, \\
&\,\; X^8Y^4X^2+X^6Y^2X^2Y^2X^2,\, X^4Y^4X^6Y^2+X^2Y^2X^8Y^2,\, \\
&\,\; X^{10}Y^4+X^8Y^2X^2Y^2+X^6Y^2X^4Y^2]^t
\end{align*}
and
$$
G_{14,4}=
\left[
\begin{array}{rrrrrrrr}
        7&0&0&0&0&0&7&7 \\
        0&7&7&0&7&7&0&0 \\
        0&7&14&0&7&7&0&0 \\
        0&0&0&7&7&7&7&7 \\
        0&7&7&7&14&14&7&7 \\
        0&7&7&7&14&14&7&7 \\
        7&0&0&7&7&7&14&14 \\
        7&0&0&7&7&7&14&14
\end{array}\right].
$$
Then $S_{14,4}(X^2,Y^2)\csim\bar v_{14,4}^* G_{14,4}\bar v_{14,4}$.
The matrix $G_{14,4}$ is positive semidefinite
with Cholesky decomposition $G_{14,4}=L_{14,4}^* L_{14,4} $, where
$$
L_{14,4}=\sqrt 7 \left[\begin{array}{rrrrrrrr}
        {1}&0&0&0&0&0&{1}&{1} \\
        0&{1}&{1}&0&{1}&{1}&0&0 \\
        0&0&{1}&0&0&0&0&0 \\
        0&0&0&{1}&{1}&{1}&{1}&{1} \\
\end{array}\right].
$$

We now consider $S_{14,6}(X^2,Y^2)$. Let $A_{14,6}$ be the symmetric
$15\times 15$ matrix from page
\pageref{mat15} and
\begin{align*}
\bar u_{14,6}=&[Y^3  X^6  Y^2  X^2  Y,\, Y  X^2  Y^2  X^2 Y^2  X^4  Y,\, 
Y^3  X^4  Y^2  X^4  Y,\, Y  X^2  Y^4  X^6  Y,\, 
 \\ &\,\;
Y^3  X^2  Y^2  X^6  Y,\, Y^5  X^8  Y,\, Y  X^4  Y^4  X^4  Y,\, Y  X^2  Y^2  X^4  Y^2  X^2  Y,\, 
Y^3  X^8  Y^3,\,  \\ &\,\;
Y  X^8  Y^5,\,
Y  X^6  Y^2  X^2 Y^3,\, Y  X^6  Y^4  X^2  Y,\,Y  X^4  Y^2  X^4  Y^3,\, \\ &\,\; Y  X^4  Y^2  X^2  Y^2  X^2  Y,\, 
Y  X^2  Y^2  X^6  Y^3]^t.
\end{align*} From the matrices on pages \pageref{mat35a} and \pageref{mat35b} we form
a symmetric $35\times 35$ matrix $B_{14,6}$ as follows: The top left $18\times 19$ block is given by the matrix on page \pageref{mat35a}, the bottom left $17\times 19$ block is given on page \pageref{mat35b} and the other entries are obtained from
$$
[B_{14,6}]_{i,j}=[B_{14,6}]_{36-j,36-i} \quad \text{for}\quad i,j>19.
$$
Let 
\begin{align*}
\bar w_{14,6}=&[Y^2  X^2  Y^2  X^6  Y^2,\, Y^4  X^8  Y^2,\, Y^2 X^6  Y^4  X^2,\, Y^2  X^4  Y^2  X^2  Y^2  X^2,\, X^2  Y^4  X^4  Y^2  X^2,\, \\ &\,\; 
       Y^2  X^2  Y^2  X^4  Y^2  X^2,\, Y^4  X^6  Y^2 X^2,\,  X^2  Y^2  X^2  Y^4  X^4,\,Y^2  X^4  Y^4  X^4,\, \\ &\,\; 
        X^2  Y^4  X^2  Y^2  X^4,\, Y^2 X^2  Y^2  X^2  Y^2  X^4,\, Y^4  X^4  Y^2  X^4,\, X^2  Y^6  X^6,\, Y^2  X^2  Y^4  X^6,\, \\ &\,\; 
        Y^4 X^2  Y^2  X^6,\, Y^6  X^8,\, X^4  Y^6  X^4,\, X^2  Y^2  X^2  Y^2 X^2  Y^2  X^2,\, Y^2  X^4  Y^2  X^4  Y^2,\,\\ &\,\; 
        X^8  Y^6,\, X^6  Y^2 X^2  Y^4,\, X^6  Y^4  X^2  Y^2,\, X^6  Y^6  X^2,\, X^4  Y^2  X^4  Y^4,\, \\ &\,\; 
       X^4 Y^2  X^2  Y^2  X^2  Y^2,\, X^4  Y^2  X^2  Y^4  X^2,\, X^4  Y^4 X^4  Y^2,\, X^4  Y^4  X^2  Y^2  X^2,\, \\ &\,\; 
       X^2  Y^2  X^6  Y^4,\, X^2  Y^2  X^4 Y^2  X^2  Y^2,\, X^2  Y^2  X^4  Y^4  X^2,\, X^2  Y^2 X^2  Y^2  X^4  Y^2,\,\\ &\,\; 
        X^2  Y^4  X^6  Y^2,\,      Y^2  X^8  Y^4,\, Y^2  X^6  Y^2  X^2  Y^2]^t
\end{align*}
Then 
\begin{equation}\label{e14-6}
S_{14,6}(X^2,Y^2)\csim \bar u_{14,6}^* A_{14,6} \bar u_{14,6} + 
\bar w_{14,6}^* B_{14,6} \bar w_{14,6}.
\end{equation}
Both matrices $A_{14,6}$ and $B_{14,6}$ are positive semidefinite as is easily checked
by looking at the corresponding characteristic polynomials using symbolic computation. 
Hence $S_{14,6}(X^2,Y^2)\in\Th^2$. By Theorem \ref{hillar},
this proves the BMV conjecture
for $m\leq 13$.

\begin{remark}
The word vectors $\bar u_{14,6}$ and $\bar w_{14,6}$ as well as the matrices
on pages \pageref{mat15}, \pageref{mat35a} and \pageref{mat35b} can be found
in the Mathematica notebook that is available with the electronic version of
the source of this article:
\begin{center}
\url{http://arxiv.org/abs/0710.1074}
\end{center}
In the same
file we also provide code that verifies the nonnegativity certificate \eqref{e14-6}
when executed.
\end{remark}

{\small
\begin{landscape}
$$
\label{mat15}
\left[
\begin{array}{rrrrrrrrrrrrrrr}
	\tabfrac{737}{45}&-\tabfrac{13}{10}&\tabfrac{247}{90}&\tabfrac{497}{90}&\tabfrac{497}{90}&\tabfrac{497}{90}&-\tabfrac{8}{5}&\tabfrac{7}{2}&\tabfrac{122}{35}&\tabfrac{3}{19}&\tabfrac{3}{19}&\tabfrac{3}{19}&\tabfrac{746}{243}&\tabfrac{413}{180}&\tabfrac{199}{32}
\\ &&&&&&&&&&&&&&\\
	-\tabfrac{13}{10}&\tabfrac{7}{5}&-\tabfrac{7}{15}&-\tabfrac{19}{30}&-\tabfrac{19}{30}&-\tabfrac{19}{30}&\tabfrac{221}{162}&\tabfrac{17}{10}&-\tabfrac{17}{20}&\tabfrac{5}{4}&\tabfrac{5}{4}&\tabfrac{5}{4}&-\tabfrac{43}{81}&-\tabfrac{48}{49}&\tabfrac{413}{180}
\\ &&&&&&&&&&&&&&\\
	\tabfrac{247}{90}&-\tabfrac{7}{15}&3&0&0&0&-\tabfrac{31}{6}&-\tabfrac{392}{81}&\tabfrac{5377}{1215}&\tabfrac{175}{972}&\tabfrac{175}{972}&\tabfrac{175}{972}&\tabfrac{1437}{500}&-\tabfrac{43}{81}&\tabfrac{746}{243}
\\ &&&&&&&&&&&&&&\\
	\tabfrac{497}{90}&-\tabfrac{19}{30}&0&\tabfrac{7}{3}&\tabfrac{7}{3}&\tabfrac{7}{3}&\tabfrac{235}{247}&\tabfrac{227}{90}&-\tabfrac{13}{90}&-\tabfrac{1}{2}&-\tabfrac{1}{2}&-\tabfrac{1}{2}&\tabfrac{175}{972}&\tabfrac{5}{4}&\tabfrac{3}{19}
\\ &&&&&&&&&&&&&&\\
	\tabfrac{497}{90}&-\tabfrac{19}{30}&0&\tabfrac{7}{3}&\tabfrac{7}{3}&\tabfrac{7}{3}&\tabfrac{235}{247}&\tabfrac{227}{90}&-\tabfrac{13}{90}&-\tabfrac{1}{2}&-\tabfrac{1}{2}&-\tabfrac{1}{2}&\tabfrac{175}{972}&\tabfrac{5}{4}&\tabfrac{3}{19}
\\ &&&&&&&&&&&&&&\\
	\tabfrac{497}{90}&-\tabfrac{19}{30}&0&\tabfrac{7}{3}&\tabfrac{7}{3}&\tabfrac{7}{3}&\tabfrac{235}{247}&\tabfrac{227}{90}&-\tabfrac{13}{90}&-\tabfrac{1}{2}&-\tabfrac{1}{2}&-\tabfrac{1}{2}&\tabfrac{175}{972}&\tabfrac{5}{4}&\tabfrac{3}{19}
\\ &&&&&&&&&&&&&&\\
	-\tabfrac{8}{5}&\tabfrac{221}{162}&-\tabfrac{31}{6}&\tabfrac{235}{247}&\tabfrac{235}{247}&\tabfrac{235}{247}&\tabfrac{2251}{200}&\tabfrac{2251}{200}&-\tabfrac{18211}{2240}&\tabfrac{235}{247}&\tabfrac{235}{247}&\tabfrac{235}{247}&-\tabfrac{31}{6}&\tabfrac{221}{162}&-\tabfrac{8}{5}
\\ &&&&&&&&&&&&&&\\
	\tabfrac{7}{2}&\tabfrac{17}{10}&-\tabfrac{392}{81}&\tabfrac{227}{90}&\tabfrac{227}{90}&\tabfrac{227}{90}&\tabfrac{2251}{200}&\tabfrac{3902}{225}&-\tabfrac{373}{45}&\tabfrac{227}{90}&\tabfrac{227}{90}&\tabfrac{227}{90}&-\tabfrac{392}{81}&\tabfrac{17}{10}&\tabfrac{7}{2}
\\ &&&&&&&&&&&&&&\\
	\tabfrac{122}{35}&-\tabfrac{17}{20}&\tabfrac{5377}{1215}&-\tabfrac{13}{90}&-\tabfrac{13}{90}&-\tabfrac{13}{90}&-\tabfrac{18211}{2240}&-\tabfrac{373}{45}&\tabfrac{712}{105}&-\tabfrac{13}{90}&-\tabfrac{13}{90}&-\tabfrac{13}{90}&\tabfrac{5377}{1215}&-\tabfrac{17}{20}&\tabfrac{122}{35}
\\ &&&&&&&&&&&&&&\\
	\tabfrac{3}{19}&\tabfrac{5}{4}&\tabfrac{175}{972}&-\tabfrac{1}{2}&-\tabfrac{1}{2}&-\tabfrac{1}{2}&\tabfrac{235}{247}&\tabfrac{227}{90}&-\tabfrac{13}{90}&\tabfrac{7}{3}&\tabfrac{7}{3}&\tabfrac{7}{3}&0&-\tabfrac{19}{30}&\tabfrac{497}{90}
\\ &&&&&&&&&&&&&&\\
	\tabfrac{3}{19}&\tabfrac{5}{4}&\tabfrac{175}{972}&-\tabfrac{1}{2}&-\tabfrac{1}{2}&-\tabfrac{1}{2}&\tabfrac{235}{247}&\tabfrac{227}{90}&-\tabfrac{13}{90}&\tabfrac{7}{3}&\tabfrac{7}{3}&\tabfrac{7}{3}&0&-\tabfrac{19}{30}&\tabfrac{497}{90}
\\ &&&&&&&&&&&&&&\\
	\tabfrac{3}{19}&\tabfrac{5}{4}&\tabfrac{175}{972}&-\tabfrac{1}{2}&-\tabfrac{1}{2}&-\tabfrac{1}{2}&\tabfrac{235}{247}&\tabfrac{227}{90}&-\tabfrac{13}{90}&\tabfrac{7}{3}&\tabfrac{7}{3}&\tabfrac{7}{3}&0&-\tabfrac{19}{30}&\tabfrac{497}{90}
\\ &&&&&&&&&&&&&&\\
	\tabfrac{746}{243}&-\tabfrac{43}{81}&\tabfrac{1437}{500}&\tabfrac{175}{972}&\tabfrac{175}{972}&\tabfrac{175}{972}&-\tabfrac{31}{6}&-\tabfrac{392}{81}&\tabfrac{5377}{1215}&0&0&0&3&-\tabfrac{7}{15}&\tabfrac{247}{90}
\\ &&&&&&&&&&&&&&\\
	\tabfrac{413}{180}&-\tabfrac{48}{49}&-\tabfrac{43}{81}&\tabfrac{5}{4}&\tabfrac{5}{4}&\tabfrac{5}{4}&\tabfrac{221}{162}&\tabfrac{17}{10}&-\tabfrac{17}{20}&-\tabfrac{19}{30}&-\tabfrac{19}{30}&-\tabfrac{19}{30}&-\tabfrac{7}{15}&\tabfrac{7}{5}&-\tabfrac{13}{10}
\\ &&&&&&&&&&&&&&\\
	\tabfrac{199}{32}&\tabfrac{413}{180}&\tabfrac{746}{243}&\tabfrac{3}{19}&\tabfrac{3}{19}&\tabfrac{3}{19}&-\tabfrac{8}{5}&\tabfrac{7}{2}&\tabfrac{122}{35}&\tabfrac{497}{90}&\tabfrac{497}{90}&\tabfrac{497}{90}&\tabfrac{247}{90}&-\tabfrac{13}{10}&\tabfrac{737}{45}
\end{array}
\right]
$$
\end{landscape}
}

{\small
\renewcommand\arraystretch{1.7}%
\begin{landscape}
$$
\label{mat35a}
\left[
\begin{array}{rrrrrrrrrrrrrrrrrrrrrrrrrrrrrrrrrrrrrrrrrrrrrrrrrrr}
	9&9&5&7&0&\tabfrac{7}{3}&\tabfrac{7}{3}&0&\tabfrac{11}{3}&0&\tabfrac{7}{3}&\tabfrac{7}{3}&0&\tabfrac{7}{3}&\tabfrac{7}{3}&\tabfrac{7}{3}&0&0&\tabfrac{5}{2} \\
	9&9&5&7&0&\tabfrac{7}{3}&\tabfrac{7}{3}&0&\tabfrac{11}{3}&0&\tabfrac{7}{3}&\tabfrac{7}{3}&0&\tabfrac{7}{3}&\tabfrac{7}{3}&\tabfrac{7}{3}&0&0&\tabfrac{5}{2} \\
	5&5&28&\tabfrac{7}{2}&5&7&7&\tabfrac{16}{19}&-\tabfrac{21}{2}&\tabfrac{16}{19}&7&7&\tabfrac{16}{19}&7&7&7&2&4&\tabfrac{13}{3} \\
	7&7&\tabfrac{7}{2}&\tabfrac{6349}{200}&14&7&7&8&11&\tabfrac{373}{90}&7&7&\tabfrac{373}{90}&7&7&7&\tabfrac{22}{9}&2&\tabfrac{23}{2} \\
	0&0&5&14&25&\tabfrac{7}{2}&\tabfrac{7}{2}&\tabfrac{1066}{81}&\tabfrac{7}{2}&\tabfrac{3494}{741}&\tabfrac{7}{2}&\tabfrac{7}{2}&\tabfrac{3494}{741}&\tabfrac{7}{2}&\tabfrac{7}{2}&\tabfrac{7}{2}&\tabfrac{85}{27}&7&0
\\
	\tabfrac{7}{3}&\tabfrac{7}{3}&7&7&\tabfrac{7}{2}&7&7&\tabfrac{7}{2}&7&\tabfrac{7}{2}&7&7&\tabfrac{7}{2}&7&7&7&1&\tabfrac{7}{2}&\tabfrac{14}{3} \\
	\tabfrac{7}{3}&\tabfrac{7}{3}&7&7&\tabfrac{7}{2}&7&7&\tabfrac{7}{2}&7&\tabfrac{7}{2}&7&7&\tabfrac{7}{2}&7&7&7&1&\tabfrac{7}{2}&\tabfrac{14}{3} \\
	0&0&\tabfrac{16}{19}&8&\tabfrac{1066}{81}&\tabfrac{7}{2}&\tabfrac{7}{2}&21&10&8&\tabfrac{7}{2}&\tabfrac{7}{2}&8&\tabfrac{7}{2}&\tabfrac{7}{2}&\tabfrac{7}{2}&18&\tabfrac{3}{2}&-\tabfrac{5}{2}
\\
	\tabfrac{11}{3}&\tabfrac{11}{3}&-\tabfrac{21}{2}&11&\tabfrac{7}{2}&7&7&10&28&6&7&7&6&7&7&7&6&4&-\tabfrac{1}{2} \\
	0&0&\tabfrac{16}{19}&\tabfrac{373}{90}&\tabfrac{3494}{741}&\tabfrac{7}{2}&\tabfrac{7}{2}&8&6&5&\tabfrac{7}{2}&\tabfrac{7}{2}&5&\tabfrac{7}{2}&\tabfrac{7}{2}&\tabfrac{7}{2}&\tabfrac{11}{2}&-\tabfrac{1}{4}&0
\\
	\tabfrac{7}{3}&\tabfrac{7}{3}&7&7&\tabfrac{7}{2}&7&7&\tabfrac{7}{2}&7&\tabfrac{7}{2}&7&7&\tabfrac{7}{2}&7&7&7&1&\tabfrac{7}{2}&\tabfrac{14}{3} \\
	\tabfrac{7}{3}&\tabfrac{7}{3}&7&7&\tabfrac{7}{2}&7&7&\tabfrac{7}{2}&7&\tabfrac{7}{2}&7&7&\tabfrac{7}{2}&7&7&7&1&\tabfrac{7}{2}&\tabfrac{14}{3} \\
	0&0&\tabfrac{16}{19}&\tabfrac{373}{90}&\tabfrac{3494}{741}&\tabfrac{7}{2}&\tabfrac{7}{2}&8&6&5&\tabfrac{7}{2}&\tabfrac{7}{2}&5&\tabfrac{7}{2}&\tabfrac{7}{2}&\tabfrac{7}{2}&\tabfrac{11}{2}&-\tabfrac{1}{4}&0
\\
	\tabfrac{7}{3}&\tabfrac{7}{3}&7&7&\tabfrac{7}{2}&7&7&\tabfrac{7}{2}&7&\tabfrac{7}{2}&7&7&\tabfrac{7}{2}&7&7&7&1&\tabfrac{7}{2}&\tabfrac{14}{3} \\
	\tabfrac{7}{3}&\tabfrac{7}{3}&7&7&\tabfrac{7}{2}&7&7&\tabfrac{7}{2}&7&\tabfrac{7}{2}&7&7&\tabfrac{7}{2}&7&7&7&1&\tabfrac{7}{2}&\tabfrac{14}{3} \\
	\tabfrac{7}{3}&\tabfrac{7}{3}&7&7&\tabfrac{7}{2}&7&7&\tabfrac{7}{2}&7&\tabfrac{7}{2}&7&7&\tabfrac{7}{2}&7&7&7&1&\tabfrac{7}{2}&\tabfrac{14}{3} \\
	0&0&2&\tabfrac{22}{9}&\tabfrac{85}{27}&1&1&18&6&\tabfrac{11}{2}&1&1&\tabfrac{11}{2}&1&1&1&\tabfrac{7396}{315}&-\tabfrac{11}{3}&-\tabfrac{16}{3} \\
	0&0&4&2&7&\tabfrac{7}{2}&\tabfrac{7}{2}&\tabfrac{3}{2}&4&-\tabfrac{1}{4}&\tabfrac{7}{2}&\tabfrac{7}{2}&-\tabfrac{1}{4}&\tabfrac{7}{2}&\tabfrac{7}{2}&\tabfrac{7}{2}&-\tabfrac{11}{3}&\tabfrac{52}{3}&-5
\end{array}
\right]
$$
\end{landscape}
}
{\tiny
\renewcommand\arraystretch{2.5}%
\begin{landscape}
$$
\label{mat35b}
\left[
\begin{array}{rrrrrrrrrrrrrrrrrrrrrrrrrrrrrrrrrrrrrrrrrrrrrrrrrrr}
	\tabfrac{5}{2}&\tabfrac{5}{2}&\tabfrac{13}{3}&\tabfrac{23}{2}&0&\tabfrac{14}{3}&\tabfrac{14}{3}&-\tabfrac{5}{2}&-\tabfrac{1}{2}&0&\tabfrac{14}{3}&\tabfrac{14}{3}&0&\tabfrac{14}{3}&\tabfrac{14}{3}&\tabfrac{14}{3}&-\tabfrac{16}{3}&-5&28
\\
	-\tabfrac{77}{90}&-\tabfrac{77}{90}&-\tabfrac{7}{6}&\tabfrac{9}{5}&1&-2&-2&1&-\tabfrac{10}{3}&-2&-2&-2&-2&-2&-2&-2&1&\tabfrac{7}{2}&\tabfrac{14}{3} \\
	-\tabfrac{77}{90}&-\tabfrac{77}{90}&-\tabfrac{7}{6}&\tabfrac{9}{5}&1&-2&-2&1&-\tabfrac{10}{3}&-2&-2&-2&-2&-2&-2&-2&1&\tabfrac{7}{2}&\tabfrac{14}{3} \\
	-\tabfrac{77}{90}&-\tabfrac{77}{90}&-\tabfrac{7}{6}&\tabfrac{9}{5}&1&-2&-2&1&-\tabfrac{10}{3}&-2&-2&-2&-2&-2&-2&-2&1&\tabfrac{7}{2}&\tabfrac{14}{3} \\
	0&0&-1&-1&-\tabfrac{13}{4}&-2&-2&1&-\tabfrac{31}{27}&-2&-2&-2&-2&-2&-2&-2&\tabfrac{11}{2}&-\tabfrac{1}{4}&0 \\
	-\tabfrac{77}{90}&-\tabfrac{77}{90}&-\tabfrac{7}{6}&\tabfrac{9}{5}&1&-2&-2&1&-\tabfrac{10}{3}&-2&-2&-2&-2&-2&-2&-2&1&\tabfrac{7}{2}&\tabfrac{14}{3} \\
	-\tabfrac{77}{90}&-\tabfrac{77}{90}&-\tabfrac{7}{6}&\tabfrac{9}{5}&1&-2&-2&1&-\tabfrac{10}{3}&-2&-2&-2&-2&-2&-2&-2&1&\tabfrac{7}{2}&\tabfrac{14}{3} \\
	0&0&-1&-1&-\tabfrac{13}{4}&-2&-2&1&-\tabfrac{31}{27}&-2&-2&-2&-2&-2&-2&-2&\tabfrac{11}{2}&-\tabfrac{1}{4}&0 \\
	-\tabfrac{28829}{4480}&-\tabfrac{28829}{4480}&-\tabfrac{55591}{20007}&-8&0&-\tabfrac{10}{3}&-\tabfrac{10}{3}&\tabfrac{7}{2}&-\tabfrac{757}{81}&-\tabfrac{31}{27}&-\tabfrac{10}{3}&-\tabfrac{10}{3}&-\tabfrac{31}{27}&-\tabfrac{10}{3}&-\tabfrac{10}{3}&-\tabfrac{10}{3}&6&4&-\tabfrac{1}{2}
\\
	0&0&\tabfrac{9}{2}&-\tabfrac{229}{81}&-\tabfrac{1327}{972}&1&1&\tabfrac{109987}{10080}&\tabfrac{7}{2}&1&1&1&1&1&1&1&18&\tabfrac{3}{2}&-\tabfrac{5}{2} \\
	-\tabfrac{77}{90}&-\tabfrac{77}{90}&-\tabfrac{7}{6}&\tabfrac{9}{5}&1&-2&-2&1&-\tabfrac{10}{3}&-2&-2&-2&-2&-2&-2&-2&1&\tabfrac{7}{2}&\tabfrac{14}{3} \\
	-\tabfrac{77}{90}&-\tabfrac{77}{90}&-\tabfrac{7}{6}&\tabfrac{9}{5}&1&-2&-2&1&-\tabfrac{10}{3}&-2&-2&-2&-2&-2&-2&-2&1&\tabfrac{7}{2}&\tabfrac{14}{3} \\
	0&0&\tabfrac{99031}{13440}&-\tabfrac{44}{3}&-\tabfrac{1240243}{162000}&1&1&-\tabfrac{1327}{972}&0&-\tabfrac{13}{4}&1&1&-\tabfrac{13}{4}&1&1&1&\tabfrac{85}{27}&7&0 \\
	-\tabfrac{413}{180}&-\tabfrac{413}{180}&\tabfrac{1369}{180}&-\tabfrac{195323}{22050}&-\tabfrac{44}{3}&\tabfrac{9}{5}&\tabfrac{9}{5}&-\tabfrac{229}{81}&-8&-1&\tabfrac{9}{5}&\tabfrac{9}{5}&-1&\tabfrac{9}{5}&\tabfrac{9}{5}&\tabfrac{9}{5}&\tabfrac{22}{9}&2&\tabfrac{23}{2}
\\
	1&1&6&\tabfrac{1369}{180}&\tabfrac{99031}{13440}&-\tabfrac{7}{6}&-\tabfrac{7}{6}&\tabfrac{9}{2}&-\tabfrac{55591}{20007}&-1&-\tabfrac{7}{6}&-\tabfrac{7}{6}&-1&-\tabfrac{7}{6}&-\tabfrac{7}{6}&-\tabfrac{7}{6}&2&4&\tabfrac{13}{3}
\\
	-\tabfrac{2246}{315}&-\tabfrac{2246}{315}&1&-\tabfrac{413}{180}&0&-\tabfrac{77}{90}&-\tabfrac{77}{90}&0&-\tabfrac{28829}{4480}&0&-\tabfrac{77}{90}&-\tabfrac{77}{90}&0&-\tabfrac{77}{90}&-\tabfrac{77}{90}&-\tabfrac{77}{90}&0&0&\tabfrac{5}{2}
\\
	-\tabfrac{2246}{315}&-\tabfrac{2246}{315}&1&-\tabfrac{413}{180}&0&-\tabfrac{77}{90}&-\tabfrac{77}{90}&0&-\tabfrac{28829}{4480}&0&-\tabfrac{77}{90}&-\tabfrac{77}{90}&0&-\tabfrac{77}{90}&-\tabfrac{77}{90}&-\tabfrac{77}{90}&0&0&\tabfrac{5}{2}
\end{array}
\right]
$$
\end{landscape}
}
\renewcommand\arraystretch{1.1}

\section{Concluding remarks}\label{conclude}

\subsection{Current state of the BMV conjecture}

The following table shows the examples we have computed on an 
ordinary PC running Mathematica with the NCAlgebra package \cite{hms}, Yalmip \cite{lof} and the SDP solver SeDuMi \cite{stu}. Most of the computations took a few seconds, some of them a few minutes.

\begin{center}
\begin{tabular}{lc}
&
\newsavebox{\foo}
\savebox{\foo}{\parbox{0.68in}{$k=m-1$ $k=m$ \textcolor{white}{$k=m$}
}}%
\begin{turn}{117}\usebox{\foo}\end{turn}

\newsavebox{\fooo}
\savebox{\fooo}{\parbox{0.4in}{$k=0$ $k=1$ $k=2$
}}%
\begin{turn}{63}\usebox{\fooo}\end{turn}
\\
$\;m$\;\vline\\
\hline
\hfill 0\;\vline& \textcolor{light-gray}{  $+$ } \\
\hfill 1\;\vline& \textcolor{light-gray}{  $+$  $+$ } \\
\hfill 2\;\vline& \textcolor{light-gray}{  $+$  $+$  $+$ } \\
\hfill 3\;\vline& \textcolor{light-gray}{  $+$  $+$  $+$  $+$ } \\
\hfill 4\;\vline& \textcolor{light-gray}{  $+$  $+$  $+$  $+$  $+$}  \\
\hfill 5\;\vline& \textcolor{light-gray}{  $+$  $+$  $+$  $+$  $+$  $+$}  \\
\hfill 6\;\vline& \textcolor{light-gray}{  $+$  $+$  $+$} $\ominus$  \textcolor{light-gray}{ $+$  $+$  $+$  }\\
\hfill 7\;\vline& \textcolor{light-gray}{  $+$  $+$  $+$}  $\oplus$  $\oplus$  \textcolor{light-gray}{  $+$  $+$  $+$ } \\
\hfill 8\;\vline& \textcolor{light-gray}{  $+$  $+$  $+$} $-$ $\oplus$ $-$ \textcolor{light-gray}{  $+$  $+$  $+$  }\\
\hfill 9\;\vline& \textcolor{light-gray}{  $+$  $+$  $+$} $-$ $\oplus$  $\oplus$ $-$  \textcolor{light-gray}{ $+$  $+$  $+$}  \\
\hfill 10\;\vline&\textcolor{light-gray}{  $+$  $+$  $+$} $-$ $\oplus$ $-$ $\oplus$ $-$ \textcolor{light-gray}{  $+$  $+$  $+$  }\\
\hfill 11\;\vline&\textcolor{light-gray}{  $+$  $+$  $+$}  $+$  $+$ $-$ $-$  $+$  $+$  \textcolor{light-gray}{  $+$  $+$  $+$  }\\
\hfill 12\;\vline&\textcolor{light-gray}{  $+$  $+$  $+$} $-$ $+$ $-$ $-$ $-$  $+$ $-$  \textcolor{light-gray}{ $+$  $+$  $+$  }\\
\hfill 13\;\vline&\textcolor{light-gray}{  $+$  $+$  $+$} $-$ $+$ $-$ $-$ $-$ $-$ $+$ $-$ \textcolor{light-gray}{  $+$  $+$  $+$  }\\
\hfill 14\;\vline&\textcolor{light-gray}{  $+$  $+$  $+$} $-$ $\oplus$ $-$ $\oplus$ $-$ $\oplus$ $-$ $\oplus$ $-$  \textcolor{light-gray}{ $+$  $+$  $+$  }\\
\hfill 15\;\vline&\textcolor{light-gray}{  $+$  $+$  $+$} $-$ $+$ $-$ $-$ $-$  $-$ $-$ $-$  $+$ $-$ \textcolor{light-gray}{  $+$  $+$  $+$  }\\
\hfill 16\;\vline&\textcolor{light-gray}{  $+$  $+$  $+$} $-$ $+$ $-$ $-$ $-$ $-$ $-$ $-$ $-$ $+$ $-$ \textcolor{light-gray}{  $+$  $+$  $+$  }\\
\hfill 17\;\vline&\textcolor{light-gray}{  $+$  $+$  $+$} $-$ $+$ $-$ $-$ $-$  $-$ $-$ $-$ $-$  $-$  $+$ $-$  \textcolor{light-gray}{ $+$  $+$  $+$  }\\
\hfill 18\;\vline&\textcolor{light-gray}{  $+$  $+$  $+$} $-$ $+$ $-$ $-$ $-$ ${}\mathbin ?{}$ ${}\mathbin ?{}$ ${}\mathbin ?{}$ $-$ $-$ $-$ $+$ $-$ \textcolor{light-gray}{  $+$  $+$  $+$  }\\
\hfill 19\;\vline&\textcolor{light-gray}{ $+$  $+$  $+$} $-$ $+$ $-$ $-$ $-$  ${}\mathbin ?{}$ ${}\mathbin ?{}$ ${}\mathbin ?{}$ ${}\mathbin ?{}$ $-$ $-$ $-$ $+$ $-$ \textcolor{light-gray}{  $+$  $+$  $+$  }\\
\end{tabular}\label{triangle}
\end{center}
\vskip 0.15cm
\centerline{\small Is $S_{m,k}(X^2,Y^2)\in\Th^2$?}

\bigskip
\begin{center}
\begin{tabular}
{c|l}
symbol & \quad meaning \\
\hline
\textcolor{light-gray}{$+$} & $S_{m,k}$ is in $\Th^2$ for trivial reasons \\
$\oplus$ & $S_{m,k}$ is in $\Th^2$ (with proof) \\
$+$ & $S_{m,k}$ is in $\Th^2$ (numerical evidence) \\
$\ominus$ & $S_{m,k}$ is not in $\Th^2$ (with proof) \\
$-$ & $S_{m,k}$ is not in $\Th^2$ (numerical evidence)
\end{tabular}
\end{center}
\vskip 0.05cm
\centerline{\small Legend}

\bigskip
While finishing our paper, Landweber and Speer sent us a closely
related preprint \cite{ls2}  where they prove for example that
$S_{m,4}(X^2,Y^2)\in\Th^2$ for odd $m$ and that $S_{11,3}(X^2,Y^2)\in\Th^2$.
Their certificates only use words from $V_1$ of Proposition \ref{wred}.
They also give results on the negative side, which imply by
Proposition \ref{speer} below that
$S_{m,k}(X^2,Y^2)\not\in\Th^2$ in the following cases:
\begin{enumerate}
\item $m$ is odd and $5\leq k\leq m-5$; 
\item $m\geq 13$ is odd and $k=3$; 
\item $m$ is even, $k$ is odd and 
$3\leq k\leq m-3$;
\item $(m,k)=(9,3)$. 
\end{enumerate}
The compatibility between our setup and the setup of Landweber and
Speer \cite{ls2} is
provided by the following proposition
communicated to us by Eugene Speer. We thank him
for letting us include this result.

\begin{proposition}\label{speer}
Retain the notation from Proposition {\rm \ref{wred}} and assume that
$m$ or $k$ is odd. Then $S_{m,k}(X^2,Y^2)\in\Th^2$ if and only if
$S_{m,k}(X^2,Y^2)\csim\bar v_1^*G_1\bar v_1$ for some positive
semidefinite $G_1$ $($or equivalently, 
if and only if 
$S_{m,k}(X^2,Y^2)\csim\bar v_2^*G_2\bar v_2$ for some positive
semidefinite $G_2$$)$.
\end{proposition}
\begin{proof}
One direction is trivial and for the converse suppose that
$S_{m,k}(X^2,Y^2)\in\Th^2$. Then by Proposition \ref{wred},
$S_{m,k}(X^2,Y^2)\csim\sum_{i=1}^2 \bar v_i^*G_i\bar v_i$
for some positive semidefinite $G_1$, $G_2$. 
Note that $w\in V_1$ if and only if $w^*\in V_2$. Hence,
$$
\bar v_1^*G_1\bar v_1= \sum_{v,u\in V_1} v^* (G_1)_{vu}u=
\sum_{w,z\in V_2} w (G_1')_{wz}z^*\csim 
\sum_{w,z\in V_2} z^* (G_1')_{wz}w=\bar v_2^*G_1'\bar v_2,
$$
where $G_1'$ is a positive semidefinite matrix obtained from 
$G_1$ by a relabelling of rows and columns.
Thus
$$
S_{m,k}(X^2,Y^2)\csim\sum_{i=1}^2 \bar v_i^*G_i\bar v_i
\csim\bar v_2^*(G_1'+G_2)\bar v_2
$$
and similarly $S_{m,k}(X^2,Y^2)\csim 
\bar v_1^*(G_1+G_2')\bar v_1$.
\end{proof}

Independently of the work of Landweber and Speer, the doctoral student
Burgdorf \cite{bur}, initially guided by further numerical experiments, found a
combinatorial proof of $S_{m,4}(X^2,Y^2)\in\Th^2$ for all $m$.

To summarize, the table on page \pageref{triangle} can be updated as follows:

\begin{center}
\begin{tabular}{lc}
\hfill 8\;\vline& \textcolor{light-gray}{  $+$  $+$  $+$} $\ominus$ $\oplus$ $\ominus$ \textcolor{light-gray}{  $+$  $+$  $+$  }\\
\hfill 9\;\vline& \textcolor{light-gray}{  $+$  $+$  $+$} $\ominus$ $\oplus$  $\oplus$ $\ominus$  \textcolor{light-gray}{ $+$  $+$  $+$}  \\
\hfill 10\;\vline&\textcolor{light-gray}{  $+$  $+$  $+$} $\ominus$ $\oplus$ $\ominus$ $\oplus$ $\ominus$ \textcolor{light-gray}{  $+$  $+$  $+$  }\\
\hfill 11\;\vline&\textcolor{light-gray}{  $+$  $+$  $+$}  $\oplus$  $\oplus$ $\ominus$ $\ominus$  $\oplus$  $\oplus$  \textcolor{light-gray}{  $+$  $+$  $+$  }\\
\hfill 12\;\vline&\textcolor{light-gray}{  $+$  $+$  $+$} $\ominus$ $\oplus$ $\ominus$ $-$ $\ominus$  $\oplus$ $\ominus$  \textcolor{light-gray}{ $+$  $+$  $+$  }\\
\hfill 13\;\vline&\textcolor{light-gray}{  $+$  $+$  $+$} $\ominus$ $\oplus$ $\ominus$ $\ominus$ $\ominus$ $\ominus$ $\oplus$ $\ominus$ \textcolor{light-gray}{  $+$  $+$  $+$  }\\
\hfill 14\;\vline&\textcolor{light-gray}{  $+$  $+$  $+$} $\ominus$ $\oplus$ $\ominus$ $\oplus$ $\ominus$ $\oplus$ $\ominus$ $\oplus$ $\ominus$  \textcolor{light-gray}{ $+$  $+$  $+$  }\\
\hfill 15\;\vline&\textcolor{light-gray}{  $+$  $+$  $+$} $\ominus$ $\oplus$ $\ominus$ $\ominus$ $\ominus$  $\ominus$ $\ominus$ $\ominus$  $\oplus$ $\ominus$ \textcolor{light-gray}{  $+$  $+$  $+$  }\\
\hfill 16\;\vline&\textcolor{light-gray}{  $+$  $+$  $+$} $\ominus$ $\oplus$ $\ominus$ $-$ $\ominus$ $-$ $\ominus$ $-$ $\ominus$ $\oplus$ $\ominus$ \textcolor{light-gray}{  $+$  $+$  $+$  }\\
\hfill 17\;\vline&\textcolor{light-gray}{  $+$  $+$  $+$} $\ominus$ $\oplus$ $\ominus$ $\ominus$ $\ominus$  $\ominus$ $\ominus$ $\ominus$ $\ominus$  $\ominus$  $\oplus$ $\ominus$  \textcolor{light-gray}{ $+$  $+$  $+$  }\\
\hfill 18\;\vline&\textcolor{light-gray}{  $+$  $+$  $+$} $\ominus$ $\oplus$ $\ominus$ $-$ $\ominus$ ${}\mathbin ?{}$ $\ominus$ ${}\mathbin ?{}$ $\ominus$ $-$ $\ominus$ $\oplus$ $\ominus$ \textcolor{light-gray}{  $+$  $+$  $+$  }\\
\hfill 19\;\vline&\textcolor{light-gray}{ $+$  $+$  $+$} $\ominus$ $\oplus$ $\ominus$ $\ominus$ $\ominus$ 
$\ominus$ $\ominus$ ${}\ominus$ $\ominus$ $\ominus$ $\ominus$ $\ominus$ $\oplus$ $\ominus$
\textcolor{light-gray}{  $+$  $+$  $+$  }\\
\hfill 20\;\vline&\textcolor{light-gray}{  $+$  $+$  $+$} $\ominus$ $\oplus$ $\ominus$ ${}\mathbin ?{}$ $\ominus$ ${}\mathbin ?{}$ $\ominus$ ${}\mathbin ?{}$ $\ominus$ ${}\mathbin ?{}$ $\ominus$ ${}\mathbin ?{}$ $\ominus$ $\oplus$ $\ominus$ \textcolor{light-gray}{  $+$  $+$  $+$  }\\
\hfill 21\;\vline&\textcolor{light-gray}{ $+$  $+$  $+$} $\ominus$ $\oplus$ $\ominus$ $\ominus$ $\ominus$ $\ominus$ $\ominus$ $\ominus$ 
$\ominus$ $\ominus$ ${}\ominus$ $\ominus$ $\ominus$ $\ominus$ $\ominus$ $\oplus$ $\ominus$ 
\textcolor{light-gray}{  $+$  $+$  $+$  }\\
\hfill 22\;\vline&\textcolor{light-gray}{ $+$  $+$  $+$} $\ominus$ $\oplus$ $\ominus$ ${}\mathbin ?{}$ $\ominus$ ${}\mathbin ?{}$ $\ominus$ ${}\mathbin ?{}$ $\ominus$ ${}\mathbin ?{}$ $\ominus$ ${}\mathbin ?{}$ $\ominus$ ${}\mathbin ?{}$ $\ominus$ $\oplus$ $\ominus$ \textcolor{light-gray}{  $+$  $+$  $+$  }\\
\end{tabular}
\end{center}
Moreover, the table continues like one would expect from looking at the lines
$m=19$, $20$, $21$, $22$. 
Hillar's descent Theorem \ref{hillar} together with
positive results for $k=4$ (by Landweber and Speer and, independently,
by Burgdorf) establishes Conjecture \ref{bmv} for $k\leq 4$ and
$m-k\leq 4$. Also,
there is still the possibility of proving the BMV
conjecture in the same manner 
by replacing a suitable sequence of $\mathbin ?$,
which only occur for even $m$ and $k$, by $\oplus$.

Very recently, using analytical methods, Fleischhack \cite{fle} and,
independently, Friedland \cite{fri} have shown the following: 
For \emph{fixed} positive semidefinite $A,B$ and 
$k\in\N$ there is an $m'\geq k$,
such that $\tr S_{m,k}(A,B)\geq 0$ for all $m\geq m'$.
If $m'$ could be chosen \emph{independently} of $A, B$, then
Conjecture \ref{bmv} would follow by Hillar's descent theorem.

\subsection{Relation to Connes' embedding conjecture}

In \cite{ksconnes} we studied the following conditions 
for real symmetric polynomials $f$ in noncommuting variables $\x:=(X_1,\ldots,X_r)$:
\begin{enumerate}[\rm (i)]
\item
$\tr(f(A_1,\ldots,A_r))\geq 0$ for all $n\in\N$ and all
$A_i\in\sym\R^{n\times n}$ with $\|A_i\|\leq 1$;\label{connes-fdim}
\item
$\tau(f(a_1,\ldots,a_r))\geq 0$ for all II$_1$-factors $\mathcal F$ and all 
$a_i \in\sym \mathcal F$ with $\|a_i\|\leq 1$;\label{connes-II1}
\item
$\forall \,\ep\in\R_{>0}\; \exists\, g\in\R\ax$: $$f+\ep \csim g\in M:=\{ \sum_i g_i^*g_i+\sum_{i,j}h_{ij}^*(1-X_i^2)h_{ij}
\mid g_i,h_{i,j}\in\R\ax\}.$$\label{connes-qm}
\end{enumerate}

We proved that \eqref{connes-II1} and \eqref{connes-qm} are equivalent and
imply \eqref{connes-fdim}. Moreover, we showed that 
the converse implication \eqref{connes-fdim} $\Rightarrow$
\eqref{connes-II1} is equivalent to an old conjecture of Connes about type II$_1$-factors. 

In Example \ref{non-sos} we have seen that $S_{6,3}(X^2,Y^2)\not\in\Th^2$, hence
the tracial version of Helton's sum of hermitian squares theorem \cite{hel} fails (cf.~also Remark \ref{helton-trace}).
By homogeneity, even $S_{6,3}(X^2,Y^2)+\ep\not\in\Th^2$ for all $\ep\in\R$.
Similarly, there is no $g\in M$ with $S_{6,3}(X^2,Y^2)\csim g$ although $S_{6,3}(X^2,Y^2)$ satisfies
\eqref{connes-fdim}. However, it is unknown whether $S_{6,3}(X^2,Y^2)$ satisfies
\eqref{connes-II1} (or equivalently, \eqref{connes-qm}).
If it does not, then Connes' embedding conjecture fails.

\subsection*{Acknowledgments.}
{\small We would like to thank Christopher Hillar for introducing the second author
to the BMV conjecture at an IMA workshop in Minneapolis. The main part of the
work was done at the Universit\"at Konstanz, the former host institution of the
second author, during a stay of the first author financed by the DFG. Part of the work was also done during the Real Algebraic Geometry workshop in
Oberwolfach in March 2007. A preliminary report of this work appeared in the
Oberwolfach reports \cite{ksmfo}.
We would like to thank Peter Landweber and Eugene Speer for the careful 
reading of a previous version of the manuscript. They provided us with a
detailed list of comments and corrections that improved the exposition as
well as some of the results. Also, Pierre Moussa and Sabine Burgdorf
contributed some valuable remarks. Finally, we would like to thank two
anonymous referees for their suggestions
which greatly contributed 
to the
overall presentation.
}

\appendix

\section{Euler-Lagrange equations}

Hillar's proof of the descent Theorem \ref{hillar} relies on 
\cite[Corollary 3.6]{hil}.
In this section we prove a similar statement, Lemma \ref{lagrange}, 
which can alternatively be used to prove
the descent theorem by a simple inspection of Hillar's proof.

Our proof of Lemma \ref{lagrange} uses only Lagrange multipliers 
and is shorter and
simpler than Hillar's variational proof of \cite[Corollary 3.6]{hil}.
However, the two results are not entirely reconcilable.

For a variational approach to the original form of the BMV
conjecture, we refer the reader to \cite{lec}, see also \cite{mou}.

\begin{lemma}\label{lagrange}
Given $n\in\N$, suppose that $(A,B)$ minimizes $\tr(S_{m,k}(A^2,B^2))$
among all symmetric $A,B\in\R^{n\times n}$ of Hilbert-Schmidt norm $1$. Suppose further
that $A$ and $B$ are positive semidefinite. Then
\begin{eqnarray}\label{lagrange-eq1}
AS_{m-1,k}(A^2,B^2)&=&\frac{m-k}m\tr(S_{m,k}(A^2,B^2))A\qquad\text{and}\\
\label{lagrange-eq2}
BS_{m-1,k-1}(A^2,B^2)&=&\frac km\tr(S_{m,k}(A^2,B^2))B.
\end{eqnarray}
\end{lemma}

\begin{proof}
We actually prove more. We fix an arbitrary $B\in\sym\R^{n\times n}$ and show
that \eqref{lagrange-eq1} holds when a positive semidefinite matrix $A$ minimizes
$\tr(S_{m,k}(A^2,B^2))$ among all $A\in\sym\R^{n\times n}$ with $\|A\|_{\rm HS}=1$.
Then a corresponding statement will hold for \eqref{lagrange-eq2} by symmetry.
Recall that the Hilbert-Schmidt norm on $\sym\R^{n\times n}$ is induced by the scalar product
given by $\langle A,B\rangle_{\rm HS}:=\tr(AB)=\sum_{i,j}A_{i,j}B_{i,j}$. 
We use the method of Lagrange multipliers and therefore compute the first
derivatives of the functions $f,g:\sym\R^{n\times n}\to\R$ given by
$$f:A\mapsto\tr(S_{m,k}(A^2,B^2))\qquad\text{and}\qquad
  g:A\mapsto\tr(A^2)=\|A\|_{\rm HS}^2.$$
The derivatives $Df(A)[H]$ and $Dg(A)[H]$ at $A\in\sym\R^{n\times n}$ along the
direction $H\in\sym\R^{n\times n}$ are the coefficients of the linear terms of
$f(A+\la H)$ and $g(A+\la H)$ considered as polynomials in $\la$, respectively.
Since
$$g(A+\la H)=\tr((A+\la H)(A+\la H))=\tr(A^2)+\la(\tr(AH)+\tr(HA))+\la^2\tr(H^2),$$
we get $Dg(A)[H]=\tr(AH)+\tr(HA)=\tr(2AH)=\langle 2A,H\rangle$, i.e., the gradient
of $g$ in $A$ is $\nabla g(A)=2A$.

The calculation of $Df(A)[H]$ is more complicated,
but follows the same scheme, namely that one occurrence of $A^2$ at a time
can be replaced by $AH$ or $HA$. The idea is the same as in the proof of \cite[Lemma 2.1]{hil}.
We have
\begin{align*}
0&=\tr\left(\sum_{i=1}^m(A^2+tB^2)^{i-1}((AH+HA)-(AH+HA))(A^2+tB^2)^{m-i}\right)\\
&=
\tr\left(m(AH+HA)(A^2+tB^2)^{m-1}\right)- \\
&\quad\, \tr\left(\sum_{i=1}^m(A^2+tB^2)^{i-1}(AH+HA)(A^2+tB^2)^{m-i}\right)
\end{align*}
and the coefficient of $t^k$ in the last expression is
$$\tr(m(AH+HA)S_{m-1,k}(A^2,B^2))-Df(A)[H].$$
This implies
$$Df(A)[H]=\langle m(AS_{m-1,k}(A^2,B^2)+S_{m-1,k}(A^2,B^2)A),H\rangle$$
and therefore $\nabla f(A)=m(AS_{m-1,k}(A^2,B^2)+S_{m-1,k}(A^2,B^2)A)$.

If $A$ is now a minimizer as stated, then we
obtain a Lagrange multiplier $\mu\in\R$ such that
$\nabla f(A)=\mu\nabla g(A)$ (since $\nabla g(A)=2A\neq 0$), i.e.,
\begin{equation}\label{precomeq}
AS_{m-1,k}(A^2,B^2)+S_{m-1,k}(A^2,B^2)A=\mu A.
\end{equation}
We now subtract the two equations that can be obtained from \eqref{precomeq} by
multiplication
with $A$ from the left and right, respectively, and see that $A^2$ commutes with
$S_{m-1,k}(A^2,B^2)$. If $A$ is in addition positive semidefinite, then also
$A$ commutes with $S_{m-1,k}(A^2,B^2)$.
Therefore \eqref{precomeq} becomes $AS_{m-1,k}(A^2,B^2)=\frac\mu 2A$.
Moreover, $$\frac\mu 2=\tr(\frac\mu 2A^2)=\tr(A^2S_{m-1,k}(A^2,B^2))=
\frac{m-k}m\tr(S_{m,k}(A^2,B^2))$$ by \cite[Lemma 2.1]{hil}.
\end{proof}

\section{Self-contained proof of Conjecture \ref{bmv} for $m=13$}

Instead of Hillar's descent Theorem \ref{hillar} one can use special
features of the found nonnegativity certificates for $S_{14,4}(X^2,Y^2)$
and $S_{14,6}(X^2,Y^2)$ to deduce Conjecture \ref{bmv} for
$m$ \emph{equal to} $13$.
We include this since the ideas might be helpful in future algebraic
approaches to the BMV conjecture.

\medskip
Retain the notation from Section \ref{export}.
From the Cholesky decomposition of $G_{14,4}$ we deduce that
\begin{equation*}
S_{14,4}(X^2,Y^2)\csim\sum_{i=1}^4 g_i^*g_i
\end{equation*}
for 
\begin{align*}
g_1 &= \sqrt 7 (Y^2X^{10}Y^2 + X^4Y^4X^6Y^2+X^2Y^2X^8Y^2+ X^{10}Y^4+X^8Y^2X^2Y^2+ \\
&\qquad\;\;\;\, X^6Y^2X^4Y^2), \\
g_2 &= \sqrt 7 ( X^4Y^2X^2Y^2X^4 + X^6Y^4X^4 + X^4Y^2X^4Y^2X^2 +  X^8Y^4X^2 + \\
&\qquad\;\;\;\,
X^6Y^2X^2Y^2X^2), \\
g_3 &= \sqrt 7 X^6Y^4X^4, \\
g_4 &= \sqrt 7 (X^2Y^2X^6Y^2X^2 + X^4Y^2X^4Y^2X^2 + X^8Y^4X^2+X^6Y^2X^2Y^2X^2 + \\
&\qquad\;\;\;\, X^4Y^4X^6Y^2+X^2Y^2X^8Y^2 + X^{10}Y^4+X^8Y^2X^2Y^2+X^6Y^2X^4Y^2).
\end{align*}

We now turn to $S_{14,6}(X^2,Y^2)$. Let $[1]_{35\times 35}$ be the 
$35\times 35$ matrix with all entries equal to $1$. Then
$B_{14,6} - \lambda [1]_{35\times 35}$ is positive semidefinite whenever 
$$
\lambda\leq\frac{5888894501020664034438572773247271387}{6345100314096416989598091089889990510969779} \approx
9.281\cdot 10^{-7}.
$$
As $\bar w_{14,6}^* [1]_{35\times 35} \bar w_{14,6} = 
S_{7,3}(X^2,Y^2)^2$, this implies that for some $h_i\in\R\ax$,
\begin{equation}\label{keycsimrelation}
S_{14,6}(X^2,Y^2)\csim 10^{-7} S_{7,3}(X^2,Y^2)^2+ \sum_i h_i^*h_i.
\end{equation}

We are now ready to prove Conjecture \ref{bmv} for 
$m=13$. It is easy to see that
$S_{13,k}(X^2,Y^2)\in\Th^2$ for $k\in\{0,1,2,11,12,13\}$.
Let us consider $S_{13,3}(A^2,B^2)$ for positive semidefinite $A,B\in\R^{n\times n}$.
Suppose there are such $A,B$ with 
\begin{equation}\label{tr133}
\tr(S_{13,3}(A^2,B^2))<0. 
\end{equation}
By Lemma \ref{lagrange},
we may without loss of generality assume that $A$ and $B$ satisfy
\eqref{lagrange-eq1} and \eqref{lagrange-eq2}
(with $m=13$ and $k=3$). Then $AS_{12,3}(A^2,B^2)$ and
$BS_{12,2}(A^2,B^2)$ are negative semidefinite, $A$ commutes with
$S_{12,3}(A^2,B^2)$ and $B$ commutes with $S_{12,2}(A^2,B^2)$. Hence
$$
S_{13,3}(A^2,B^2)=A^2S_{12,3}(A^2,B^2)+B^2S_{12,2}(A^2,B^2)
$$
is negative semidefinite and so is $BS_{13,3}(A^2,B^2)B$. By the above,
$S_{14,4}(X^2,Y^2)\in \Th^2$, so
$$
0\leq\tr(S_{14,4}(A^2,B^2))=\frac{14}{10} \tr(B^2S_{13,3}(A^2,B^2))=\frac{14}{10} \tr(BS_{13,3}(A^2,B^2)B)\leq 0.
$$
(For the first equality see e.g.~\cite[Lemma 2.1]{hil}.)
As $S_{14,4}(X^2,Y^2)\csim\sum_{i=1}^4 g_i^*g_i$ with $g_3=\sqrt 7 X^6Y^4X^4$ and
$\tr(S_{14,4}(A^2,B^2))=0$, $A^6B^4A^4=0$ by Lemma \ref{trrad}. In particular, $\tr((B^2A^{5})^*(B^2A^{5}))=0$,
hence $B^2A^{5}=0$. Repeating this we obtain $BA^{5/2}=A^{5/2}B=0$. But then
$S_{13,3}(A^2,B^2)=0$, contradicting \eqref{tr133}. This proves the BMV conjecture
for $(m,k)\in\{(13,3),(13,10)\}$. Similarly, the cases $(m,k)=(13,4)$ and $(m,k)=(13,9)$ can be handled.

Let us now consider $S_{13,5}(A^2,B^2)$ for positive semidefinite $A,B\in\R^{n\times n}$.
Suppose there are such $A,B$ with 
\begin{equation}\label{tr135}
\tr(S_{13,5}(A^2,B^2))<0. 
\end{equation}
As before, we can deduce that $\tr(S_{14,6}(A^2,B^2))=0$. 
From \eqref{keycsimrelation} it follows that $S_{7,3}(A^2,B^2)=0$. 
By \eqref{sos73}, 
this implies $B^2A^4B=0$, thus $B^{3/2}A^2=A^2 B^{3/2}=0$. Therefore $S_{13,5}(A^2,B^2)=0$,
contradicting \eqref{tr135}.
This settles Conjecture \ref{bmv} for $(m,k)\in\{(13,5),(13,8)\}$. To conclude the proof we note that
the two remaining cases $(m,k)=(13,6)$ and $(m,k)=(13,7)$ can be handled similarly.

\end{document}